\documentclass[11pt,a4paper]{article}
\usepackage[cp1251]{inputenc}
\usepackage{amsmath,amsthm}
\usepackage{amsfonts,amstext,amssymb,verbatim,epsfig}

\def\R{{\mathbb{R}}}

\def\Z{{\mathbb{Z}}}

\renewcommand{\P}{\mathbb{P}}

\newtheorem{theorem}{Theorem}
\newtheorem{lemma}{Lemma}

\newtheoremstyle{likedef}
  {}%
  {}%
  {}%
  {\parindent}%
  {\bfseries}%
  {.}%
  {.5em}%
  {}%

\theoremstyle{likedef}
\newtheorem{definition}{Definition}[section]
\newtheorem{remark}{Remark}

\numberwithin{equation}{section}

\begin{document}

\title{Connectivity properties of random interlacement and\\ intersection of random walks}

\author{Bal\'azs R\'ath\thanks{ETH Z\"urich, Department of Mathematics, R\"amistrasse 101, 8092 Z\"urich. Email: balazs.rath@math.ethz.ch and artem.sapozhnikov@math.ethz.ch.
The research of both authors has been supported by the grant ERC-2009-AdG  245728-RWPERCRI.}
\and
Art\"{e}m Sapozhnikov\footnotemark[1]
}

\maketitle

\footnotetext{MSC2000: Primary 60K35, 82B43.}
\footnotetext{Keywords: Random interlacement; random walk; intersection of random walks; capacity; Wiener test.}

\begin{abstract}
We consider the interlacement Poisson point process on the space of doubly-infinite $\Z^d$-valued trajectories modulo time shift, tending to infinity at positive and negative infinite times. 
The set of vertices and edges visited by at least one of these trajectories is the random interlacement at level $u$ of Sznitman \cite{SznitmanAM}. 
We prove that for any $u>0$, almost surely, 
(1) any two vertices in the random interlacement at level $u$ are connected via at most $\lceil d/2\rceil$ trajectories of the point process, and 
(2) there are vertices in the random interlacement at level $u$ which can only be connected via at least $\lceil d/2\rceil$ trajectories of the point process. 
In particular, this implies the already known result of \cite{SznitmanAM} that the random interlacement at level $u$ is connected. 
\end{abstract}

\section{Introduction}
\noindent

The model of random interlacements was introduced by Sznitman in \cite{SznitmanAM} 
in order to describe the local picture left by the trajectory of a random walk on
the discrete torus $(\Z/N\Z)^d$, $d \geq 3$ when it runs up to times of order $N^d$,
or on the discrete cylinder $(\Z/N\Z)^d \times \Z$ , $d \geq 2$, 
when it runs up to times of order $N^{2d}$, see \cite{sznitman_cylinders}, \cite{windisch_torus}. 
Informally, the random interlacement Poisson point process consists of a countable collection of doubly infinite trajectories on $\Z^d$, 
and the trace left by these trajectories on a finite subset of $\Z^d$ ``looks like'' the trace of the above mentioned random walks. 

In this paper we investigate connectivity properties of the random interlacement, 
giving a detailed picture about how the collection of doubly infinite trajectories are actually interlaced.
Our methods are further developed in \cite{RS:resistance,RS:percolation} to show, respectively, that 
in dimensions $d\geq 3$ for all $u>0$, the random interlacement graph is almost surely transient, 
and the Bernoulli percolation on it has a non-trivial phase transition (even in wide enough slabs). 
Recently, J. \v{C}ern\'y and S. Popov \cite{CP:distance} used the techniques of this paper 
to prove the shape theorem for the graph distance on random interlacements. 

An essentially different proof of the main result of this paper was obtained by E. B. Procaccia and J. Tykesson \cite{PT}. 
It involves ideas of stochastic dimension theory developed in \cite{BKPS}.

\subsection{The model}\label{sec:model}
\noindent

Let $W$ be the space of doubly-infinite nearest-neighbor trajectories in $\Z^d$ ($d\geq 3$) which tend to infinity at positive and negative infinite times, and 
let $W^*$ be the space of equivalence classes of trajectories in $W$ modulo time shift. 
We write $\mathcal W$ for the canonical $\sigma$-algebra on $W$ generated by the coordinates $X_n$, $n\in\Z$, and 
$\mathcal W^*$ for the largest $\sigma$-algebra on $W^*$ for which the canonical map $\pi^*$ from $(W,\mathcal W)$ to $(W^*,\mathcal W^*)$ is measurable. 
Let $u$ be a positive number. 
We say that a Poisson point measure $\mu$ on $W^*$ has distribution $\mathrm{Pois}(u,W^*)$ if the following properties hold: 
For a finite subset $A$ of $\Z^d$, let $\mu_A$ be the restriction of $\mu$ to the set of trajectories from $W^*$ that intersect $A$, and 
let $N_A$ be the number of trajectories in $\mathrm{Supp}(\mu_A)$. 
Then $\mu_A = \sum_{i=1}^{N_A}\delta_{\pi^*(X_i)}$, where $X_i$ are doubly-infinite trajectories from $W$ parametrized in such a way that $X_i(0) \in A$ and $X_i(t) \notin A$ for all $t<0$ and for all $i\in\{1,\ldots,N_A\}$, and 
\begin{itemize}
\item[(1)]\label{distr:1}
The random variable $N_A$ has Poisson distribution with parameter $u\mathrm{cap}(A)$ (see \eqref{eq:defcap} for the definition of the $\mathrm{cap}(A)$). 
\item[(2)]\label{distr:2}
Given $N_A$, the points $X_i(0)$, $i\in\{1,\ldots,N_A\}$, are independent and distributed according to the normalized equilibrium measure on $A$ (see \eqref{eq:defneqm} for the definition). 
\item[(3)]\label{distr:3}
Given $N_A$ and $(X_i(0))_{i=1}^{N_A}$, the corresponding forward and backward paths are conditionally independent, $(X_i(t), t\geq 0)_{i=1}^{N_A}$ are distributed as independent simple random walks, and $(X_i(t), t\leq 0)_{i=1}^{N_A}$ are distributed as independent random walks conditioned on not hitting $A$.
\end{itemize}
Properties (1)-(3) uniquely define $\mathrm{Pois}(u,W^*)$ as proved in Theorem~1.1 in \cite{SznitmanAM}. 
In fact, Theorem~1.1 in \cite{SznitmanAM} gives a coupling of the Poisson point measures $\mu(u)$ with distribution $\mathrm{Pois}(u,W^*)$ for all $u>0$, but 
we will not need such a general statement here.  
We also mention a couple of properties of the distribution $\mathrm{Pois}(u,W^*)$, which will be useful in the proofs. 
Property (4) follows from the above definition of $\mathrm{Pois}(u,W^*)$, and (5) is a property of Poisson point measures. 
\begin{enumerate}
\item[(4)]\label{distr:4}
Let $\mu_1$ and $\mu_2$ be independent Poisson point measures on $W^*$ with distributions $\mathrm{Pois}(u_1,W^*)$ and $\mathrm{Pois}(u_2,W^*)$, respectively. 
Then $\mu_1+\mu_2$ has distribution $\mathrm{Pois}(u_1+u_2,W^*)$. 
\item[(5)]\label{distr:5}
Let $S_1,\ldots,S_k$ be disjoint elements of $\mathcal W^*$. 
We denote by $I(S_i)\mu$ the restriction of $\mu$ to the set of trajectories from $S_i$. 
Then $I(S_1)\mu,\ldots,I(S_k)\mu$ are independent Poisson point measures on $W^*$.
\end{enumerate}
We refer the reader to \cite{SznitmanAM} for more details. 
For a Poisson point measure $\mu$ with distribution $\mathrm{Pois}(u,W^*)$, the random interlacement $\mathcal I$ at level $u$ is defined as 
\begin{equation}\label{eq:definterlacement}
\mathcal I = \mathcal I(\mu) = \bigcup_{w\in\mathrm{Supp}(\mu)}\mathrm{range}(w). 
\end{equation}

\subsection{The result}
\noindent

We consider a random point measure $\mu$ on $W^*$ distributed as $\mathrm{Pois}(u,W^*)$. 
We denote by $\mathbb P$ the law of $\mu$. 
Our main result concerns the geometric properties of the support of $\mu$. 
Remember that the support of $\mu$ consists of a countable set of doubly-infinite random walk trajectories modulo time shift. 
We construct the random graph $G=(V,E)$ as follows. The set of vertices $V$ is the set of trajectories from $\mathrm{Supp}(\mu)$, 
and the set of edges $E$ is the set of pairs of different trajectories from $\mathrm{Supp}(\mu)$ that intersect. 
Let $\mathrm{diam}(G)$ be the diameter of $G$. 
Our main result is the following theorem. 
\begin{theorem}\label{thm:diameter}
For $d\geq 3$, let 
\begin{equation}\label{def:sd}
s_d = \lceil (d-2)/2\rceil ,\
\end{equation}
where $\lceil a\rceil$ is the smallest integer not less than $a$. 
Then
\[
\mathbb P\left(\mathrm{diam}(G) = s_d\right) = 1 ,\
\]
In particular, we get an alternative proof of (2.21) in \cite{SznitmanAM}, which states that the random interlacement $\mathcal I$ 
is a connected subgraph of $\Z^d$. 
\end{theorem}

\begin{remark}\label{rem:d34}
In dimensions $3$ and $4$, the result is a trivial consequence of Theorem~2.6 in \cite{Lawler34} (see also remark at the bottom of page 661 in \cite{Lawler34}) which states that two independent random walks in dimension $3$ or $4$ intersect infinitely often with probability $1$. 
Therefore, it remains to prove the theorem for $d\geq 5$. 
\end{remark}

The structure of the proof of Theorem~\ref{thm:diameter} can be non-rigorously summarized as follows:
first we pick one of the doubly infinite trajectories from $\mathrm{Supp}(\mu)$. 
Denote by $\mathcal{A}^{(1)}$ the set of vertices of $\Z^d$ visited by this trajectory. 
The second layer $\mathcal{A}^{(2)}$ consists of the vertices visited by those trajectories of $\mathrm{Supp}(\mu)$ that intersect 
$\mathcal{A}^{(1)}$, and recursively let $\mathcal{A}^{(s)}$ denote the set of vertices visited by the trajectories that intersect 
$\mathcal{A}^{(s-1)}$. We prove that $\P(\mathrm{diam}(G) = s_d)=1$ by showing that, almost surely, $\mathcal{A}^{(s_d)} \neq \mathcal{I}$ and 
$\mathcal{A}^{(s_d+1)} = \mathcal{I}$.

Let us recall the following well-known fact (see, e.g., Proposition~2.3 in \cite{Lawler34}): For $d\geq 3$, 
the probability that a simple random walk from $0$ hits $x$ is comparable with $\min(1,|x|^{2-d})$. 
We will use this fact and the following elementary lemma to show that $\mathcal{A}^{(s_d)} \neq \mathcal I$.
\begin{lemma}\label{l:RWIntersections}
There exists a finite constant $C=C(d)$ such that for any positive integer $n$ and for any $z_0,z_{n+1}\in\Z^d$, 
\[
\sum_{z_1,\ldots,z_n\in\Z^d} \prod_{i=0}^n \min\left(1,|z_i - z_{i+1}|^{2-d}\right)
~~\begin{cases} \leq C|z_0-z_{n+1}|^{2n+2-d} &\mbox{if } n< s_d,\\
= \infty&\mbox{otherwise.}
\end{cases}
\]
\end{lemma}
(See, e.g. (1.38) of Proposition~1.7 in \cite{HHS} for a proof of Lemma~\ref{l:RWIntersections}.) 
Lemma~\ref{l:RWIntersections} gives bounds on $n$-fold convolutions of the probability that a random walk from $z_0$ ever visits $z_{n+1}$. 
We will see that $\mathbb P(0,x\in \mathcal A^{(s)})$ can be estimated as a $(s-1)$-fold convolution of such hitting probabilities, and, therefore, 
we will conclude from Lemma~\ref{l:RWIntersections} that $\mathbb P(0,x\in \mathcal A^{(s)}) \leq C|x|^{2s-d}$. 
In particular, $\mathbb P(0,x\in \mathcal A^{(s_d)})\to 0$ as $|x|\to\infty$.  
This contradicts $\mathcal A^{(s_d)} = \mathcal I$, since $\mathcal I$ has positive density. 

In order to show that $\mathcal{A}^{(s_d+1)}=\mathcal I$, we argue as follows. 
Heuristically, $\mathcal{A}^{(s)}$ is a $2s$-dimensional object as long as $2s<d$.  
The capacity of $\mathcal{A}^{(s)}$ intersected with a ball of radius $R$ (see \eqref{eq:defcap} for the definition of the capacity)
is comparable to $R^{2s}$ as long as $2s \leq d-2$. 
The set $\mathcal{A}^{(s_d)}$ already saturates the ball in terms of capacity, 
thus it is visible for an independent random walk started somewhere inside the ball of radius $R$. 
We apply a variant of Wiener's test (see, e.g., Proposition~2.4 in \cite{Lawler34}) to show that any random walk hits $\mathcal{A}^{(s_d)}$ almost surely. 

This is the general strategy of the proof. 
Instead of following it directly, we benefit from property (4) of $\mathrm{Pois}(u,W^*)$ by decomposing $\mu$ into a sum of $s_d$ i.i.d point measures $\mu^{(s)}$ with distribution $\mathrm{Pois}(u/s_d,W^*)$ and constructing each $\mathcal A^{(s)}$ from the ``new'' measure $\mu^{(s)}$. 

The paper is organized as follows. In Section~\ref{sec:notation} we collect most of the notation and facts used in the paper. 
The most important of those are the definitions and properties of the Green function and the capacity. 
We prove the lower bound of Theorem~\ref{thm:diameter} in Section~\ref{sec:lowerbound}, and the upper bound in Section~\ref{sec:upperbound}. 
The structure of the proof of the upper bound of Theorem~\ref{thm:diameter} is given at the beginning of Section~\ref{sec:upperbound}.

\section{Notation and facts about Green function and capacity}\label{sec:notation}
\noindent

In this section we collect most of the notation, definitions and facts used in the paper. 
For $a\in\R$, we write $|a|$ for the absolute value of $a$, $\lfloor a\rfloor$ for the integer part of $a$, and $\lceil a\rceil$ for the smallest integer not less than $a$. 
For $x\in \Z^d$, we write $|x|$ for $\max\left(|x_1|,\ldots,|x_d|\right)$. For a set $S$, we write $|S|$ for the cardinality of $S$. For $R>0$ and $x\in\Z^d$, let $B(x,R) = \{y\in\Z^d~:~|x-y|\leq R\}$  be the ball of radius $R$ centered at $x$. 
We denote by $I(A)$ the indicator of event $A$, and by $E[X;A]$ the expected value of random variable $XI(A)$. 
Throughout the text, we write $c$ and $C$ for small positive and large finite constants, respectively, that may depend on $d$ and $u$. 
Their values may change from place to place. 

For $x\in\Z^d$, let $P_x$ be the law of a simple random walk $X$ on $\Z^d$ with $X(0) = x$. 
We write $g(\cdot,\cdot)$ for the Green function of the walk: 
\[
g(x,y) = \sum_{t=0}^\infty P_x(X(t) = y),~~x,y\in\Z^d .\
\] 
We also write $g(\cdot)$ for $g(0,\cdot)$. The Green function is symmetric and, by translation invariance, $g(x,y) = g(y-x)$.   
It follows from \cite[Theorem~1.5.4]{LawlerRW} that for any $d\geq 3$ there exist a positive constant $c_g=c_g(d)$ and a finite constant $C_g = C_g(d)$ such that 
for all $x$ and $y$ in $\Z^d$, 
\begin{equation}\label{eq:gfbounds}
c_g\min\left(1, |x-y|^{2-d}\right)\leq g(x,y)\leq C_g\min\left(1, |x-y|^{2-d}\right) .\
\end{equation}
\begin{definition}
Let $K$ be a subset of $\Z^d$. The energy of a finite Borel measure $\nu$ on $K$ is 
\[
\mathcal E(\nu) = \int_K\int_K g(x,y) d\nu(x) d\nu(y) = \sum_{x,y\in K} g(x,y)\nu(x)\nu(y) .\
\]
The capacity of $K$ is 
\begin{equation}\label{eq:defcap}
\mathrm{cap}(K) = \left[\inf_{\nu} \mathcal E(\nu)\right]^{-1} ,\
\end{equation}
where the infimum is over probability measures $\nu$ on $K$. (We assume that $\infty^{-1} = 0$, i.e. the capacity of the empty set is $0$.)
\end{definition}
The following properties of the capacity immediately follow from \eqref{eq:defcap}: 
\begin{eqnarray}
\mbox{Monotonicity:}
&&\mbox{for any}~K_1\subset K_2\subset \Z^d, ~~\mathrm{cap}(K_1)\leq \mathrm{cap}(K_2) ;\label{cap:monotonicity}\\
\mbox{Subadditivity:}
&&\mbox{for any}~K_1, K_2\subset \Z^d, ~~\mathrm{cap}(K_1\cup K_2) \leq \mathrm{cap}(K_1)+\mathrm{cap}(K_2) ;\label{cap:subadditivity}\\
\mbox{Capacity of a point:}
&&\mbox{for any}~x\in\Z^d, ~~\mathrm{cap}(\{x\}) = 1/g(0) .\label{cap:point}
\end{eqnarray}
It will be useful to have an alternative definition of the capacity in $d\geq 3$. 
\begin{definition}
Let $K$ be a finite subset of $\Z^d$. The equilibrium measure of $K$ is defined by 
\begin{equation}\label{eq:defeqm}
e_K(x) = P_x\left(X(t)\notin K~\mbox{for all}~t\geq 1\right) I(x\in K),~~x\in\Z^d .\
\end{equation}
The capacity of $K$ is then equal to the total mass of the equilibrium measure of $K$:

\[
\mathrm{cap}(K) = \sum_x e_K(x) ,\
\] 
and the unique minimizer of the variational problem \eqref{eq:defcap} is given by 
the normalized equilibrium measure 
\begin{equation}\label{eq:defneqm}
\widetilde e_K(x) = e_K(x)/\mathrm{cap}(K) .\
\end{equation}
(See, e.g., Lemma~2.3 in \cite{JainOrey} for a proof of this fact.)
\end{definition}
As a simple corollary of the above definition, we get for $d\geq 3$, 
\begin{equation}\label{eq:hittingformula}
P_x\left(H(K) < \infty\right) = \sum_{y\in K} g(x,y) e_K(y),~~\mbox{for}~x\in\Z^d .\
\end{equation}
Here, we write $H(K)$ for the first entrance time in $K$, i.e. $H(K) = \inf\{t\geq 0~:~X(t)\in K\}$. 
We will repeatedly use the following bound on the capacity of $B(0,R)$ in $d\geq 3$ (see (2.16) on page 53 in \cite{LawlerRW}): 
There exist constants $c_b = c_b(d)>0$ and $C_b = C_b(d)<\infty$ such that for all positive $R$, 
\begin{equation}\label{eq:capball}
c_b R^{d-2}\leq \mathrm{cap}\left(B(0,R)\right) \leq C_b R^{d-2} .\
\end{equation}

\section{Proof of Theorem~\ref{thm:diameter}: lower bound on the diameter}\label{sec:lowerbound}
\noindent

Remember the definition of $s_d$ in \eqref{def:sd}. 
In this section we prove that $\mathbb P(\mathrm{diam}(G)\geq s_d) = 1$. 
Since, almost surely, $\mathrm{diam}(G)\geq 1$, we only need to consider the case $d\geq 5$. 
For two trajectories $v$ and $w$ in $V$, we write $\rho(v,w)$ for the distance between $v$ and $w$ in $G$.  
In order to prove that the probability of the event $\{\mathrm{diam}(G) \geq s_d\}$ is $1$, 
we assume by contradiction that this probability is $\leq 1-\delta$, for some positive $\delta$.  
In other words, the probability of event 
\[
E = \left\{\rho(v,w)\leq s_d-1~\mbox{for all}~v,w\in V\right\}
\]
is bounded from below by $\delta$. 

For $x,y\in\Z^d$, we denote by $S(x,y)$ the subset of doubly-infinite trajectories in $W^*$ that intersect both vertices $x$ and $y$. 
Remember the definition \eqref{eq:definterlacement} of the random interlacement $\mathcal I$. 
The next lemma gives an estimate on the probability that $E$ occurs and two different vertices $x$ and $y$ of $\Z^d$ are in $\mathcal I$: 
\begin{lemma}\label{l:xyE}
For any $x,y\in\Z^d$, 
\begin{equation}\label{eq:xyE}
\mathbb P(\{x,y\in\mathcal I\}\cap E) \leq
\sum_{n=0}^{s_d-1}\sum_{z_1,\ldots,z_n\in\Z^d}
\prod_{i=0}^n \mathbb E\left[\mu\left(S(z_i,z_{i+1})\right)\right] ,\
\end{equation}
where we take $z_0 = x$ and $z_{n+1} = y$. 
\end{lemma}
We postpone the proof of Lemma~\ref{l:xyE} until the end of this section. 
Each of the expectations $\mathbb E\left[\mu\left(S(z_i,z_{i+1})\right)\right]$ in \eqref{eq:xyE} is bounded from above by $2 u g(z_i,z_{i+1})$. 
(This follows, for example, from (1.33) in \cite{SznitmanPTRF} applied to $K = \{z_i\}$ and $K' = \{z_{i+1}\}$.)
Therefore, we obtain 
\[
\mathbb P(\{x,y\in\mathcal I\}\cap E) \leq 
\sum_{n=0}^{s_d-1}(2u)^{n+1} \sum_{z_1,\ldots,z_n\in\Z^d}\prod_{i=0}^n g(z_i,z_{i+1}) ,\
\]
where we again assume $z_0 = x$ and $z_{n+1} = y$.
Recall from \eqref{eq:gfbounds} that $g(x,y)\leq C_g\min(1,|x-y|^{2-d})$. 
Therefore, by Lemma~\ref{l:RWIntersections}, 
\[
\sum_{n=0}^{s_d-1}\sum_{z_1,\ldots,z_n\in\Z^d}\prod_{i=0}^n g(z_i,z_{i+1}) \leq C|z_0-z_{n+1}|^{2s_d - d} 
\leq C |z_0-z_{n+1}|^{-1} .\
\]
In particular, $\mathbb P(\{x,y\in\mathcal I\}\cap E) \leq C|x-y|^{-1} \to 0$, as $|x-y|\to\infty$. 
By property (1) of $\mathrm{Pois}(u,W^*)$, for any $R>0$, 
\[
\mathbb P\left(\mathcal I\cap B(0,R)\neq\emptyset\right) 
= 
\mathbb P\left(N_{B(0,R)}\geq 1\right)
=
1-e^{-u\mathrm{cap}(B(0,R))} .\
\]
By \eqref{eq:capball}, we can take $R$ big enough so that 
\[
\mathbb P\left(\mathcal I\cap B(0,R)\neq\emptyset\right) \geq 1 - \frac{\delta}{3} .\
\]
With this choice of $R$, for any $z\in\Z^d$, we obtain
\[
\mathbb P\left(\{\mathcal I\cap B(0,R)\neq\emptyset\}\cap\{\mathcal I\cap B(z,R)\neq\emptyset\}\cap E\right) 
\geq \mathbb P(E) - 2\mathbb P\left(\mathcal I\cap B(0,R) = \emptyset\right)
\geq \delta/3 .\
\]
On the other hand, for $z\in\Z^d$ with $|z|>3R$, 
\[
\mathbb P\left(\{\mathcal I\cap B(0,R)\neq\emptyset\}\cap\{\mathcal I\cap B(z,R)\neq\emptyset\}\cap E\right) 
\leq \sum_{x\in B(0,R)}\sum_{y\in B(z,R)} \mathbb P(\{x,y\in\mathcal I\}\cap E)
\leq C R^{2d} |z|^{-1} ,\
\]
which tends to $0$ as $|z|$ tends to infinity. 
This is a contradiction, and we conclude that $\mathbb P$-a.s. the diameter of $G$ is at least $s_d$. 
\qed

\begin{proof}[Proof of Lemma~\ref{l:xyE}]
One can deduce the result almost immediately from the Palm theory for general Poisson point processes (see, e.g. Chapter~13.1 in \cite{DVJ}). 
Remember the definition of the set $S(x,y)$ given before the statement of Lemma~\ref{l:xyE}. 
Let $D(x,y)$ be the event that $S(x,y)\cap\mathrm{Supp}(\mu) \neq \emptyset$. In other words, $D(x,y) = \{\mu(S(x,y)) \neq 0\}$. 
For $x,y,x',y'\in\Z^d$, we write $D(x,y)\circ D(x',y')$ for the event that there exist different trajectories $w$ and $w'$ in $\mathrm{Supp}(\mu)$ such that 
$w\in S(x,y)$ and $w'\in S(x',y')$. 
Let ${\sum}^*$ be the sum over all $(n+1)$-tuples of pairwise different doubly-infinite trajectories modulo time-shift $w_0,\ldots,w_n\in\mathrm{Supp}(\mu)$. We have 
\begin{eqnarray*}
\mathbb P(\{x,y\in\mathcal I\}\cap E) 
&\leq 
&\sum_{n=0}^{s_d-1}\sum_{z_1,\ldots,z_n\in\Z^d}
\mathbb P\left(D(z_0,z_1)\circ\ldots\circ D(z_n,z_{n+1})\right) \\
&\leq 
&\sum_{n=0}^{s_d-1}\sum_{z_1,\ldots,z_n\in\Z^d}
\mathbb E \left[{\sum}^* \prod_{i=0}^n I(w_i\in S(z_i,z_{i+1}))\right] ,\
\end{eqnarray*}
where we take $z_0 = x$ and $z_{n+1} = y$. 
The result then follows from the Slivnyak-Mecke theorem (See, e.g. Theorem~3.3 in \cite{MW}, where it is proved for point processes in $\R^d$, and 
Chapter~13.1 in \cite{DVJ} for the theory of Palm distributions in general spaces.): 
\[
\mathbb E \left[{\sum}^* \prod_{i=0}^n I(w_i\in S(z_i,z_{i+1}))\right]
=
\prod_{i=0}^n \mathbb E\left[\mu\left(S(z_i,z_{i+1})\right)\right] .\
\]
\end{proof}

\section{Proof of Theorem~\ref{thm:diameter}: upper bound on the diameter}\label{sec:upperbound}
\noindent

The proof of the upper bound on the diameter of $G$ in Theorem~\ref{thm:diameter} is organized as follows. 
Section~\ref{sec:capacity} contains preliminary lemmas. 
Lemma~\ref{l:rwcap} gives some bounds on the expected capacity of a certain family of traces of random walks. 
Lemma~\ref{l:intcap} provides bounds on the expected capacity of a set of vertices visited by trajectories from $\mathrm{Supp}(\mu)$ that intersect a given set of vertices. 
Both lemmas state that the capacity of such sets of vertices is either comparable with the volume of the set (when trajectories are ``well spread-out'') or 
with the capacity of the ball that contains the set (when the set is ``dense'' in the ball). 
In Lemma~\ref{l:intrcap}, we show that the exclusion of a (small) number of trajectories from $\mathrm{Supp}(\mu)$ that visit a certain ball does not decrease too much the capacity of sets in Lemma~\ref{l:intcap}. This step is needed to benefit from property (5) of $\mathrm{Pois}(u,W^*)$ and create some additional independence. 

In Section~\ref{sec:visible} we use these bounds on the capacity to construct certain subsets of $\mathrm{Supp}(\mu)$ (see \eqref{eq:defarr1} and \eqref{eq:defarr2}) that are visible by an independent random walk started near the origin. 

In Section~\ref{sec:recurrent} we construct a sequence of almost independent visible subsets of $\mathrm{Supp}(\mu)$ and use ideas similar in spirit to Wiener's test to show that, almost surely, infinitely many of these sets are visited by an independent random walk. This is done in Lemma~\ref{l:hitting1}.

We finish Section~\ref{sec:recurrent} by completing the proof of Theorem~\ref{thm:diameter}.

\subsection{Bounds on the capacity of certain collection of random walk trajectories}\label{sec:capacity}
\noindent

\begin{lemma}\label{l:gfestimate}
Let $d\geq 5$. Let $(x_i)_{i\geq 1}$ be a sequence in $\Z^d$, and let $X_i$ be a sequence of independent simple random walks on $\Z^d$ with $X_i(0) = x_i$. 
Then for all positive integers $N$ and $n$ and for all $(x_i)_{i\geq 1}$, we have 
\begin{equation}\label{eq:gfestimate}
{\mathbf E}\left[\sum_{i,j=1}^N \sum_{s,t=n+1}^{2n} g\left(X_i(s),X_j(t)\right)\right]
\leq
C\left(Nn + N^2 n^{3-d/2}\right) .\
\end{equation}
\end{lemma}
\begin{proof}
Let $X$ be a simple random walk with $X(0) = 0$, then for all $y\in \Z^d$ and for all positive integers $s$, 
\begin{equation}\label{eq:gfsup}
{\mathbf E} g\left(X(s),y\right) \leq C s^{1-d/2} .\ 
\end{equation}
Indeed, by the Markov property, 
\[
{\mathbf E} g\left(X(s),y\right) 
= 
\sum_{t = s}^\infty {\mathbf P} \left(X(t) = y\right)
\leq 
C \sum_{t = s}^\infty t^{-d/2}
\leq 
C s^{1-d/2} .\
\]
Here we used the fact that \cite[Proposition~7.6]{Spitzer}
\[
\sup_{y\in\Z^d} {\mathbf P} \left(X(t) = y\right) \leq C t^{-d/2} .\
\]
In order to prove \eqref{eq:gfestimate}, we consider separately the cases $i=j$ and $i\neq j$. 
In the first case, the Markov property and the fact that $g(x,y)=g(x-y)$ imply 
\begin{eqnarray*}
{\mathbf E}\left[\sum_{i=1}^N \sum_{s,t=n+1}^{2n} g\left(X_i(s),X_i(t)\right)\right]
&=
&N {\mathbf E}\left[\sum_{s,t=n+1}^{2n} g\left(X(|s-t|)\right)\right]\\
&\stackrel{\eqref{eq:gfsup}}\leq 
&C N n \left(1 + \sum_{s=1}^n s^{1 - d/2}\right) 
\stackrel{(d\geq 5)}\leq 
C N n .\
\end{eqnarray*}
In the case $i\neq j$, an application of \eqref{eq:gfsup} gives 
\[
{\mathbf E}\left[\sum_{s,t=n+1}^{2n} g\left(X_i(s),X_j(t)\right)\right]
\leq
n^2 C n^{1-d/2} .\
\]
This completes the proof. 
\end{proof}

Let $(X_i(t)~:~t\geq 0)_{i\geq 1}$ be a sequence of nearest-neighbor trajectories on $\Z^d$, and 
$\overline X_N = (X_1,\ldots,X_N)$.  
For positive integers $N$ and $R$, we define the subset $\Phi(\overline X_N, R)$ of $\Z^d$ by 
\begin{equation}\label{def:Phi}
\Phi(\overline X_N, R) 
=
\bigcup_{i=1}^N \left(\left\{X_i(t)~:~1\leq t\leq R^2/2\right\}\cap B(X_i(0),R)\right) .\
\end{equation}

\begin{lemma}\label{l:rwcap}
Let $X_i$ be a sequence of independent simple random walks on $\Z^d$ with $X_i(0) = x_i$. 
There exists a positive constant $c$ such that for any sequence $(x_i)_{i\geq 1}\subset\Z^d$ and for all positive integers $N$ and $R$, 
\begin{equation}\label{eq:rwcapupper}
\mathrm{cap}\left(\Phi(\overline X_N, R)\right) \leq \frac{NR^2}{2g(0)} ,\
\end{equation}
and for $d\geq 5$, 
\begin{equation}\label{eq:rwcaplower}
\mathbf E \mathrm{cap}\left(\Phi(\overline X_N, R)\right) \geq c \min\left(N R^2, R^{d-2}\right) .\
\end{equation}
\end{lemma}
\begin{proof}
The upper bound on the capacity of $\Phi(\overline X_N, R)$ follows from properties \eqref{cap:subadditivity} and \eqref{cap:point}, and 
the fact that the number of vertices in $\Phi(\overline X_N, R)$ is at most $NR^2/2$. 

We proceed with the lower bound on $\mathbf E \mathrm{cap}\left(\Phi(\overline X_N, R)\right)$. 
The following inequality follows from Kolmogorov's maximal inequality applied coordinatewise: 
For each $\lambda > 0$ and $n\geq 1$, 
\begin{equation}\label{eq:kolmogorov}
\mathbf P \left(\max_{1\leq t \leq n} |X(t)| \geq \lambda \right) \leq \frac{n}{\lambda^2} .\
\end{equation}
Take positive integers $N$ and $R$, random walks $X_1,\ldots, X_N$ with $X_i(0) = x_i$, and set 
\begin{equation}\label{eq:ndef}
n = \lfloor R^2/4\rfloor .\ 
\end{equation}
We define the random subset $J$ of $\{1,\ldots N\}$ by
\[
J = \{i~:~\sup_{1\leq t\leq 2n} |X_i(t) - x_i| \leq R\} .\
\]
We also consider the event $A = \{|J| \geq N/4\}$. 
It follows from (\ref{eq:kolmogorov}) that 
\[
\mathbf E |J| \geq N \left(1 - \frac{2n}{R^2}\right) \geq \frac{N}{2} .\
\]
Since $|J| \leq N$, we get $\mathbf P \left(A\right) \geq \frac{1}{3}$. 

By the definition \eqref{eq:defcap} of the capacity of $\Phi(\overline X_N, R)$, we have 
\[
\mathbf E \mathrm{cap}\left(\Phi(\overline X_N, R)\right)
\geq 
\mathbf E \left[\mathcal E(\nu)^{-1}\right]
\geq
\mathbf E \left[\mathcal E(\nu)^{-1}; A\right] ,\
\]
where $\nu$ stands for the probability measure 
\[
\nu(x) = \frac{1}{|J| n}\sum_{i\in J}\sum_{t=n+1}^{2n}I(X_i(t) = x), ~~x\in\Z^d .\
\]
The energy of $\nu$ equals 
\[
\mathcal E(\nu) = \frac{1}{|J|^2 n^2}\sum_{i,j\in J}\sum_{s,t=n+1}^{2n} g(X_i(s),X_j(t)) .\
\]
Therefore, in order to prove the lower bound on $\mathbf E \mathrm{cap}\left(\Phi(\overline X_N, R)\right)$, 
it suffices to show that 
\[
\mathbf E \left[\left(\frac{1}{|J|^2 n^2}\sum_{i,j\in J}\sum_{s,t=n+1}^{2n} g(X_i(s),X_j(t))\right)^{-1}; A\right]
\geq 
c \min\left(N R^2, R^{d-2}\right) .\
\]
By the Cauchy-Schwarz inequality and the definition of the event $A$, we get 
\[
\mathbf E \left[\left(\frac{1}{|J|^2 n^2}\sum_{i,j\in J}\sum_{s,t=n+1}^{2n} g(X_i(s),X_j(t))\right)^{-1}; A\right]
\geq 
(N/4)^2 n^2 \mathbf P(A)^2 
\left(\mathbf E\left[\sum_{i,j\in J}\sum_{s,t=n+1}^{2n} g(X_i(s),X_j(t)); A\right]\right)^{-1} .\
\]
Since $J$ is a subset of $\{1,\ldots,N\}$, the right-hand side is bounded from below by 
\begin{eqnarray*}
(N/4)^2 n^2 \mathbf P(A)^2 
\left(\mathbf E\left[\sum_{i,j=1}^N \sum_{s,t=n+1}^{2n} g(X_i(s),X_j(t))\right]\right)^{-1} 
&\stackrel{\eqref{eq:gfestimate}}\geq
&\frac{N^2n^2}{144 C(Nn + N^2 n^{3-d/2})}\\
&\stackrel{\eqref{eq:ndef}}\geq 
&c \min\left(N R^2, R^{d-2}\right) .\
\end{eqnarray*}
This completes the proof. 
\end{proof}

Let $A$ be a finite set of vertices in $\Z^d$. For a point measure $\omega = \sum_{i\geq 0}\delta_{w_i}$ with $w_i\in W^*$, 
we denote by $N_A(\omega)$ the number of trajectories from $\mathrm{Supp}(\omega)$ that intersect $A$. 
(In particular, for a point measure $\mu$ with distribution $\mathrm{Pois}(u,W^*)$, we have $N_A = N_A(\mu)$.) 
Let $X_1,\ldots,X_{N_A(\omega)}$ be these trajectories parametrized in such a way that 
$X_i(0) \in A$ and $X_i(t) \notin A$ for all $t<0$ and for all $i\in\{1,\ldots,N_A(\omega)\}$. 
We write $\overline X_A(\omega)$ for $(X_1,\ldots,X_{N_A(\omega)})$. 
We also define $\Psi(\omega,A,R)$ as $\Phi(\overline X_A(\omega),R)$, i.e., 
\begin{equation}\label{def:Psi}
\Psi(\omega,A,R) \stackrel{\eqref{def:Phi}}= \Phi(\overline X_A(\omega),R) = \bigcup_{i=1}^{N_A(\omega)} \left(\left\{X_i(t)~:~1\leq t\leq R^2/2\right\}\cap B(X_i(0),R)\right) .\
\end{equation}

\begin{lemma}\label{l:intcap}
Let $d\geq 5$. Let $\mu$ be a Poisson point measure with distribution $\mathrm{Pois}(u,W^*)$, then 
for all finite subsets $A$ of $\Z^d$ and for all positive $R$, one has
\[
\mathbb E \mathrm{cap}(\Psi(\mu,A,R)) \geq c~ \min\left(u\mathrm{cap}(A) R^2, R^{d-2}\right) .\
\]
\end{lemma}
\begin{proof}
Let $\lambda = u\mathrm{cap}(A)$. 
Properties (2) and (3) of $\mathrm{Pois}(u,W^*)$ and Lemma~\ref{l:rwcap} imply that 
\[
\mathbb E \mathrm{cap}(\Psi(\mu,A,R)) \geq c~ \mathbb E \min\left(N_A R^2, R^{d-2}\right) .\
\]
Property (1) of $\mathrm{Pois}(u,W^*)$ implies that $\mathbb E N_A = \lambda$, $\mathbb E [N_A^2] = \lambda^2 + \lambda$, and $\mathbb P(N_A = 0) = \exp(-\lambda)$. 
If $\lambda\leq 1/2$, we estimate
\[
\mathbb E \min\left(N_A R^2, R^{d-2}\right)
\geq
R^2 \mathbb P(N_A\geq 1) = R^2(1-e^{-\lambda}) \geq R^2\lambda/2 .\
\]
If $\lambda\geq 1/2$, we write 
\[
\mathbb E \min\left(N_A R^2, R^{d-2}\right)
\geq
\min\left(\frac{R^2\lambda}{2},R^{d-2}\right) \mathbb P\left(N_A\geq \frac{\lambda}{2}\right) .\
\]
Remember the Paley-Zygmund inequality \cite{PZ}: 
Let $\xi$ be a non-negative random variable with finite second moment. For any $\theta\in(0,1)$, 
\begin{equation}\label{eq:PZ}
\mathrm P(\xi\geq \theta \mathrm E\xi) \geq (1-\theta)^2\frac{\left[\mathrm E \xi\right]^2}{\mathrm E[\xi^2]} .\
\end{equation}
An application of \eqref{eq:PZ} to $N_A$ gives 
\[
\mathbb P\left(N_A\geq \frac{\lambda}{2}\right) 
\geq \frac{1}{4} \frac{\lambda^2}{\lambda^2+\lambda} 
\geq \frac{1}{12} .\
\]
This completes the proof. 
\end{proof}

\begin{definition}\label{def:RIrestricted}
For positive integers $r$ and $R$ with $r < R$, and a point measure $\omega = \sum_{i\geq 0}\delta_{w_i}$ with $w_i\in W^*$, 
we write $\omega_r$ for the restriction of $\omega$ to the set of trajectories that intersect $B(r)$, 
$\omega_{r,\infty}$ for the restriction of $\omega$ to the set of trajectories that do not intersect $B(r)$, and  
$\omega_{r,R}$ for the restriction of $\omega$ to the set of trajectories that intersect $B(R)$ but do not intersect $B(r)$.
By property (5) of $\mathrm{Pois}(u,W^*)$, the measures $\mu_r$ and $\mu_{r,R}$ are independent for any $r>0$ and $R\in (r,\infty]$. 
Moreover, for any $r>0$, we have $\mu = \mu_r + \mu_{r,\infty}$. 
\end{definition}

\begin{lemma}\label{l:intrcap}
Let $d\geq 5$. 
For all finite subsets $A$ of $\Z^d$ and for all positive integers $r$ and $R$ with $r<R$, 
\[
\mathbb E \mathrm{cap}(\Psi(\mu_{r,\infty},A,R)) \geq c~ \min\left(u\mathrm{cap}(A) R^2, R^{d-2}\right) - Cur^{d-2}R^2 .\
\]
\end{lemma}
\begin{proof}
By the subadditivity of the capacity and the fact that $\mu = \mu_r + \mu_{r,\infty}$, 
\[
\mathbb E \mathrm{cap}(\Psi(\mu_{r,\infty},A,R))
\geq
\mathbb E \mathrm{cap}(\Psi(\mu,A,R)) - 
\mathbb E \mathrm{cap}(\Psi(\mu_r,A,R)) .\
\]
We use Lemma~\ref{l:intcap} to bound $\mathbb E \mathrm{cap}(\Psi(\mu,A,R))$ from below. 
As for an upper bound on $\mathbb E \mathrm{cap}(\Psi(\mu_r,A,R))$, 
note that $|\mathrm{Supp}(\mu_r)| = \mu_r(W^*) = N_{B(r)}(\mu_r) = N_{B(r)}$. 
Therefore, by Lemma~\ref{l:rwcap}, 
\[
\mathbb E \mathrm{cap}(\Psi(\mu_r,A,R)) 
\leq 
\frac{R^2 \mathbb E N_{B(r)}}{2g(0)}  = \frac{R^2u\mathrm{cap}(B(r))}{2g(0)} \stackrel{\eqref{eq:capball}}\leq Cur^{d-2}R^2 .\
\]
\end{proof}

\subsection{Construction of visible sets}\label{sec:visible}
\noindent

Let $X$ be a simple random walk on $\mathbb Z^d$ with $X(0) = x$. 
We denote the corresponding probability measure and the expectation by $P_x$ and $E_x$,  respectively. 
Let $\mu^{(2)}, \mu^{(3)}, \ldots$ be independent random point measures with distribution $\mathrm{Pois}(u,W^*)$ (The parameter $u$ is fixed here.), 
which are also independent of $X$. The corresponding probability measures and expectations are denoted by $P^{(2)}, P^{(3)}, \ldots$ and 
$E^{(2)}, E^{(3)}, \ldots$, respectively. 
For $s\geq 1$, we write $\mathbb P^{(s)}_x$ for $P_x\otimes P^{(2)}\otimes \ldots \otimes P^{(s)}$. 

Let $r$ and $R$ be positive integers with $r<R$ and $|x| < R$. 
Let $T_{B(R)}$ be the first exit time of $X$ from $B(R)$, i.e., $T_{B(R)} = \inf\{t\geq 0~:~X(t)\notin B(R)\}$. 
We denote by $Y$ the random walk $X(T_{B(R)} + \cdot)$. 
We define the following sequence of random subsets of $\Z^d$: 
\begin{equation}\label{eq:defarr1}
A^{(1)}(r,R) = A^{(1)}(R) = \Phi (Y,R)
\stackrel{\eqref{def:Phi}} = \left\{Y(t)~:~1\leq t\leq R^2/2\right\}\cap B(Y(0),R) ,\ 
\end{equation}
and for $s\geq 2$ (see \eqref{def:Psi} for notation), 
\begin{equation}\label{eq:defarr2}
A^{(s)}(r,R) = \Psi \left(\mu^{(s)}_{r,\infty}, A^{(s-1)}(r,R), R\right)
= \Psi \left(\mu^{(s)}_{r,sR}, A^{(s-1)}(r,R), R\right) ,\
\end{equation}
where the last equality follows from the fact that $A^{(s-1)}(r,R)$ is a subset of $B(sR)$ by construction. 
\begin{remark}\label{rem:distance}
Note that for each $y\in A^{(s)}(r,R)$, there exist
doubly-infinite trajectories $w_i\in\mathrm{Supp}(\mu^{(i)}_{r,\infty})$, $2\leq i\leq s$, 
such that (1) the vertex $y$ is visited by $w_s$, 
(2) the random walk $X$ intersects $w_2$, and 
(3) for all $i\in\{2,\ldots,s-1\}$, the trajectories $w_i$ and $w_{i+1}$ intersect.
\end{remark}
\begin{lemma}\label{l:aupper}
Let $s$ be a positive integer. 
There exist finite constants $C_s = C(u,d,s)$ such that for all positive integers $r$ and $R$ with $r<R$ and for all $x\in B(R)$, 
\begin{equation}\label{eq:aupper1}
\mathbb E^{(s)}_x \mathrm{cap}\left(A^{(s)}(r,R)\right) \leq C_s R^{\min(d-2,2s)} ,\
\end{equation}
and
\begin{equation}\label{eq:aupper2}
\mathbb E^{(s)}_x \left[\mathrm{cap}\left(A^{(s)}(r,R)\right)^2\right] \leq C_s R^{2\min(d-2,2s)} .\
\end{equation}
\end{lemma}
\begin{proof}
We fix $r$ and $R$ throughout the proof, and we write $A^{(s)}$ for $A^{(s)}(r,R)$. 
Since $A^{(s)}$ is a subset of $B((s+1)R)$, the monotonicity of the capacity implies that  
\[
\mathrm{cap}\left(A^{(s)}\right) \leq \mathrm{cap}\left(B((s+1)R)\right) \stackrel{\eqref{eq:capball}}\leq C_s R^{d-2} .\
\]
Therefore, it suffices to show that the first and the second moments of $\mathrm{cap}\left(A^{(s)}\right)$ are bounded from above by 
$C_sR^{2s}$ and $C_sR^{4s}$, respectively. 
It follows from \eqref{eq:rwcapupper} that 
\[
\mathbb E^{(s)}_x \mathrm{cap}\left(A^{(s)}\right)
\leq 
\frac{R^2}{2g(0)} \mathbb E^{(s)}_x N_{A^{(s-1)}}(\mu^{(s)}_{r,\infty})
\leq
\frac{R^2}{2g(0)} \mathbb E^{(s)}_x N_{A^{(s-1)}}(\mu^{(s)}) .\
\]
Remember that $N_{A^{(s-1)}}(\mu^{(s)})$ is a Poisson random variable with parameter $u\mathrm{cap}\left(A^{(s-1)}\right)$,  
therefore, we have  
\[
\mathbb E^{(s)}_x \mathrm{cap}\left(A^{(s)}\right)
\leq 
\frac{R^2}{2g(0)} \mathbb E^{(s-1)}_x u\mathrm{cap}\left(A^{(s-1)}\right) .\
\]
The bound on the first moment of $\mathrm{cap}\left(A^{(s)}\right)$ follows by induction. 
The bound on the second moment of $\mathrm{cap}\left(A^{(s)}\right)$ is also obtained using \eqref{eq:rwcapupper}. 
In a similar fashion as above, we obtain the relations:  
\[
\mathbb E^{(s)}_x \left[\mathrm{cap}\left(A^{(s)}\right)^2\right]
\leq 
\frac{R^4}{4g(0)^2} \mathbb E^{(s)}_x \left[N_{A^{(s-1)}}(\mu^{(s)}_{r,\infty})^2\right] ,\
\] 
and
\begin{eqnarray*}
\mathbb E^{(s)}_x \left[N_{A^{(s-1)}}(\mu^{(s)}_{r,\infty})^2\right]
&\leq 
&\mathbb E^{(s)}_x \left[N_{A^{(s-1)}}(\mu^{(s)})^2\right]\\
&=
&\mathbb E^{(s-1)}_x\left[u^2\mathrm{cap}\left(A^{(s-1)}\right)^2\right] + 
\mathbb E^{(s-1)}_x u\mathrm{cap}\left(A^{(s-1)}\right) .\
\end{eqnarray*}
The bound on the second moment of $\mathrm{cap}\left(A^{(s)}\right)$ follows from these inequalities and from the first statement of the lemma. 
\end{proof}

\begin{lemma}\label{l:capAlb}
Let $d\geq 5$. Let $s$ be a positive integer. 
There exist positive constants $c_s = c(u,d,s)$ and $\varepsilon=\varepsilon(u,d,s)$ such that for all positive integers $r$ and $R$ with
\begin{equation}\label{eq:rR1}
r^{d-2} \leq \varepsilon R
\end{equation} 
and for all $x\in B(R)$, 
\begin{equation}\label{eq:capAlb}
\mathbb E^{(s)}_x \mathrm{cap}\left(A^{(s)}(r,R)\right) \geq c_s R^{\min(d-2,2s)} .\
\end{equation}
\end{lemma}
\begin{remark}\label{rem:saturation}
Remember that $A^{(s)}(r,R)$ is constructed as a subset of a (random) number of pieces of random walk trajectories of lengths $\lfloor R^2/2\rfloor$. 
The expected capacity of a single random walk is comparable with its length in dimension $\geq 5$, as shown in Lemma~\ref{l:rwcap}. 
Note that $\min(d-2,2s)$ is $2s$ for $s < \lceil (d-2)/2\rceil$ and $d-2$ for $s\geq \lceil (d-2)/2\rceil$. 
One can interpret the results of Lemma~\ref{l:capAlb} as follows. 
If $s\leq \lceil (d-2)/2\rceil$, the random walk pieces that form $A^{(s)}(r,R)$ are well spread-out, so that the expected capacity of $A^{(s)}(r,R)$ is comparable with its volume. 
On the other hand, if $s\geq \lceil (d-2)/2\rceil$, the set $A^{(s)}(r,R)$ saturates the ball $B((s+1)R)$ and its expected capacity becomes comparable with the capacity of the ball, which is of order $R^{d-2}$ by \eqref{eq:capball}.  
\end{remark}
\begin{proof}
We prove \eqref{eq:capAlb} by induction on $s$. 

It follows from \eqref{eq:rwcaplower} that 
\[
E_x \mathrm{cap}\left(A^{(1)}(R)\right) \geq c_1 R^2 .\
\]
Let $s\geq 2$, and assume that the induction hypothesis holds: 
\[
\mathbb E^{(s-1)}_x \mathrm{cap}\left(A^{(s-1)}(r,R)\right) \geq c_{s-1} R^{\min(d-2,2s-2)} .\
\]
With this lower bound on the expected value of $\mathrm{cap}\left(A^{(s-1)}(r,R)\right)$ and the corresponding upper bound 
\eqref{eq:aupper2}, the Paley-Zygmund inequality \eqref{eq:PZ} yields that 
there exists a positive constant $c = c(u,d,s)$ such that 
\begin{equation}\label{eq:acapc}
\mathbb P^{(s-1)}_x \left(\mathrm{cap}\left(A^{(s-1)}(r,R)\right) \geq c R^{\min(d-2,2s-2)}\right) \geq c .\
\end{equation}
Lemma~\ref{l:intrcap} implies that 
\begin{eqnarray*}
\mathbb E^{(s)}_x \mathrm{cap}\left(A^{(s)}(r,R)\right) 
&\geq 
&c~ \mathbb E^{(s-1)}_x \min\left(u\mathrm{cap}(A^{(s-1)}(r,R)) R^2, R^{d-2}\right) - Cur^{d-2}R^2 \\
&\stackrel{\eqref{eq:acapc}}\geq
&c~\min\left(R^{\min(d-2,2s-2)} R^2, R^{d-2}\right) - Cur^{d-2}R^2 \\
&=
&c~R^{\min(d-2,2s)} - Cur^{d-2}R^2\\
&\stackrel{\eqref{eq:rR1}}\geq
&(c/2)~R^{\min(d-2,2s)} .\
\end{eqnarray*}
(The last inequality holds if $\varepsilon$ in \eqref{eq:rR1} is taken small enough, since we only consider $d\geq 5$ and $s\geq 2$.)
\end{proof}

In the next lemma we study the probability that a simple random walk hits $A^{(s)}(r,R)$. 
Remember the definitions of $X$ and $\mu^{(s)}$, $s\geq 2$ at the beginning of Section~\ref{sec:visible}, and $s_d$ in \eqref{def:sd}.  
\begin{lemma}\label{l:hittingc}
Let $d\geq 5$. 
Let $Z$ be a simple random walk on $\Z^d$ with $Z(0) = z$, which is independent of $X$ and $\mu^{(s)}$, $s\geq 2$, with law $P_z$.  
There exist positive constants $c = c(u,d)$, and $\varepsilon =\varepsilon(u,d)>0$ such that, 
for all positive integers $r$ and $R$  with $r^{d-2} \leq \varepsilon R$, $x\in B(R)$, and $z\in B(R)$, we have 
\[
P_z\otimes{\mathbb P}_x^{(s_d)}\left(H(A^{(s_d)}(r,R)) < T_{B(R^2)}\right) \geq c ,\
\]
where $H(A^{(s)}(r,R))$ is the entrance time of $Z$ in $A^{(s)}(r,R)$ and $T_{B(R^2)}$ the exit time of $Z$ from $B(R^2)$. 
\end{lemma}
\begin{proof}
We write $A$ for $A^{(s_d)}(r,R)$ throughout the proof. 
We use the identity (\ref{eq:hittingformula}):  
\[
P_z\left(H(A) < \infty\right) = \sum_{y\in A} g(z,y) e_{A}(y) ,\
\]
where $e_A$ is the equilibrium measure of $A$ (see \eqref{eq:defeqm}). 
We have  
\[
P_z\otimes{\mathbb P}^{(s_d)}_x\left(H(A) < \infty\right) 
=
{\mathbb E}^{(s_d)}_x \left[ \sum_{y\in A} g(z,y) e_{A}(y)\right] .\
\]
Note that $A$ is a subset of $B((s_d+1)R) \subset B(dR)$ by construction. 
Therefore, inequality \eqref{eq:gfbounds} implies that, for any $y\in A$ and $z\in B(R)$, $g(z,y) \geq c_g (2dR)^{2-d}$. 
Also remember that $\sum_{y\in A} e_{A}(y) = \mathrm{cap} (A)$. 
These observations give 
\[
P_z\otimes{\mathbb P}^{(s_d)}_x\left(H(A) < \infty\right) 
\geq
c_g (2dR)^{2-d} {\mathbb E}^{(s_d)}_x\left[\mathrm{cap}(A)\right] .\
\]
It follows from the previous lemma that, for $d\geq 5$, we can choose $\varepsilon >0$ so that 
\[
{\mathbb E}^{(s_d)}_x\left[\mathrm{cap}(A)\right] \geq c R^{\min(d-2,2s_d)} = c R^{d-2} .\
\]
Therefore, 
\[
P_z\otimes{\mathbb P}^{(s_d)}_x\left(H(A) < \infty\right) 
\geq c .\
\]
On the other hand, by the strong Markov property of the random walk $Z$, 
\begin{eqnarray*}
P_z\otimes{\mathbb P}^{(s_d)}_x\left( T_{B(R^2)} < H(A)<\infty\right) 
&\leq 
&\sup_{z'\notin B(R^2)}
P_{z'}\otimes{\mathbb P}^{(s_d)}_x\left(H(A) < \infty\right) \\
&\leq
&\sup_{z'\notin B(R^2)}
P_{z'}\left(H(B(dR)) < \infty\right) .\ 
\end{eqnarray*}
In the second inequality we use the fact that $A$ is a subset of $B(dR)$. 
We bound the right-hand side, using \eqref{eq:gfbounds}, \eqref{eq:hittingformula} and \eqref{eq:capball}:
\[
\sup_{z'\notin B(R^2)} P_{z'}\left(H(B(dR)) < \infty\right)
\leq 
C_g (R^2-dR)^{2-d} \mathrm{cap}\left(B(dR)\right) 
\leq 
CR^{2-d} .\
\]
Remember that $R \geq r^{d-2}/\varepsilon\geq 1/\varepsilon$. 
Therefore, by taking $\varepsilon$ small enough, we get 
\[
\sup_{z'\notin B(R^2)} P_{z'}\left(H(B(dR)) < \infty\right)
\leq 
\frac{1}{2}P_z\otimes{\mathbb P}^{(s)}_x\left(H(A^{(s)}) < \infty\right) .\
\]
The result follows. 
\end{proof}

\subsection{Construction of recurrent sets}\label{sec:recurrent}
\noindent

We will now use the result of Lemma~\ref{l:hittingc} to construct a sequence of subsets $A^{(s_d)}(r_k,R_k)$ of $\Z^d$ such that 
the union of these sets $\cup_k A^{(s_d)}(r_k,R_k)$ is hit by an independent random walk (infinitely often) with probability $1$. 
Remember the definitions of $X$ and $\mu^{(s)}$, $s\geq 2$ at the beginning of Section~\ref{sec:visible}. 

\begin{lemma}\label{l:hitting1}
Let $d\geq 5$. For $z\in\Z^d$, 
let $Z$ be a simple random walk on $\Z^d$ with $Z(0) = z$, which is independent of $X$ and $\mu^{(s)}$, $s\geq 2$. 
Let $P_z$ be its law. Let $X(0) = x$. 
There exist sequences of positive integers $r_k$ and $R_k$ such that 
\[
P_z\otimes{\mathbb P}^{(s_d)}_x \left(\limsup_k \left\{H(A^{(s_d)}(r_k,R_k)) < \infty\right\}\right) = 1 ,\
\]
where $H(A^{(s)}(r,R))$ is the entrance time of $Z$ in $A^{(s)}(r,R)$. 
\end{lemma}
\begin{proof}
Let $\varepsilon$ be the positive number from Lemma~\ref{l:hittingc}. 
We define $r_k$ and $R_k$ recursively: 
\[
r_0 = \max\left(|x|,|z|\right),~~R_0 = \lceil\varepsilon^{-1} r_0^{d-2}\rceil ,\
\]
and, for $k\geq 1$, 
\[
r_k = d R_{k-1}^2,~~R_k = \lceil\varepsilon^{-1} r_k^{d-2}\rceil .\
\]
(Any sequences that grow faster than these would do.)
We consider the following sequence of (random) subsets of $\Z^d$: 
\[
A_k = A^{(s_d)}(r_k,R_k) ~(\subset B(dR_k)).\
\]
Note that the following properties hold: 
\begin{itemize}
\item[(i)]
the set of vertices $\{X(t)~:~t\leq T_{B(R_k)}+(R_k^2/2)\}$  is contained in $B(r_{k+1})$,  
\item[(ii)]
$r_k$ and $R_k$ satisfy the assumptions of Lemma~\ref{l:hittingc}, and
\item[(iii)]
the set $A_k$ is measurable with respect to the sigma-algebra generated by $\{X(t)~:~t\leq T_{B(r_{k+1})}\}$ and $\mu^{(i)}_{r_k,r_{k+1}}$ for $2\leq i\leq s_d$.
\end{itemize}
Property (i) follows from the fact that $R_k + (R_k^2/2) < r_{k+1}$. 
Property (ii) follows from our choice of $\varepsilon$ and from the fact that $r_k^{d-2} \leq \varepsilon R_k$. 
In order to see that property (iii) holds, note that, by the definition of $A^{(s)}(r,R)$ in \eqref{eq:defarr1} and \eqref{eq:defarr2}, 
set $A^{(i-1)}(r_k,R_k)$ is contained in $B(iR_k)$. Therefore, 
set $A_k$ is measurable with respect to the sigma-algebra generated by $\{X(t)~:~t\leq T_{B(R_k)} + (R_k^2/2)\}$ and $\mu^{(i)}_{r_k,iR_k}$ for $i\leq s_d$. 
Since $s_d R_k < r_{k+1}$ and $\{X(t)~:~t\leq T_{B(R_k)}+(R_k^2/2)\}\subset B(r_{k+1})$, property (iii) follows.  
\\[5mm]

Consider the events $\Gamma_k = \{H(A_k) < T_{B(R_k^2)}\}$ and their indicator functions $\gamma_k = I(\Gamma_k)$. 
In this definition, $H(A_k)$ is the entrance time of $Z$ in $A_k$ and $T_{B(R_k^2)}$ is the exit time of $Z$ from $B(R_k^2)$. 
We will show that there exists a positive constant $c$ such that 
for all $k\geq 1$ and for any $g_{1},\ldots, g_{k-1}\in\{0,1\}$, 
\begin{equation}\label{eq:Borel}
P_z\otimes{\mathbb P}^{(s_d)}_x\left(\Gamma_k~|~\gamma_{1} = g_{1},\ldots,\gamma_{k-1} = g_{k-1}\right) \geq c > 0 .\
\end{equation}
The result will then follow from Borel's lemma \cite{Borel}:
\begin{lemma}
Consider a probability space $(\Omega,\mathcal F,\mathbb P)$ and a sequence of events 
$\Delta_n\in\mathcal F$. 
Let $\delta_n = I(\Delta_n)$ be the indicator function of the event $\Delta_n$. 
If there exists a sequence $b_n$ such that $\sum_n b_n =\infty$ and for any 
$d_i\in \{0,1\}$, $i=1,\ldots,n-1$, 
\[
\mathbb P(\Delta_n~|~\delta_1 = d_1,\ldots, \delta_{n-1} = d_{n-1}) \geq b_n > 0
\]
then
\[
\mathbb P\left(\limsup_k\Delta_k\right) = 1 .\
\]
\end{lemma}

We will now prove \eqref{eq:Borel}. We denote by $E$ the event $\{\gamma_{1} = g_{1},\ldots,\gamma_{k-1} = g_{k-1}\}$.
By property (iii) above and the fact that $\{Z(t)~:~t\leq T_{B(R_{k-1}^2)}\}\subset B(r_k)$, 
the event $E$ is measurable with respect to the sigma-algebra $\mathcal F_{k-1}$ generated by 
$\{X(t)~:~t\leq T_{B(r_k)}\}$, $\mu^{(s)}_{r_k}$ for $s\leq s_d$, and $\{Z(t)~:~t\leq T_{B(r_k)}\}$. 
(Here, the two occurrences of $T_{B(r_k)}$ correspond to the exit times of $X$ and $Z$ from $B(r_k)$, respectively, which are, in general, different.)
By property (5) of $\mathrm{Pois}(u,W^*)$, the sets of point measures $\{\mu^{(s)}_{r_k}\}_{s\geq 2}$ and $\{\mu^{(s)}_{r_k,r_{k+1}}\}_{s\geq 2}$ are independent. 
Therefore, using strong Markov property for $X$ and $Z$ and integrating over the $\mu^{(s)}_{r_k,r_{k+1}}$, $s\geq 2$, we obtain
\[
P_z\otimes{\mathbb P}^{(s_d)}_x\left(\Gamma_k\cap E\right) \geq 
E_z\otimes{\mathbb E}^{(s_d)}_x \left[P_{z'}\otimes{\mathbb P}^{(s_d)}_{x'}\left(\Gamma_k\right) ; E\right] ,\
\]
where $x' = X(T_{B(r_k)})$, and $z' = Z(T_{B(r_k)})$. 
It follows from Lemma~\ref{l:hittingc} that 
\[
P_{z'}\otimes{\mathbb P}^{(s_d)}_{x'}\left(\Gamma_k\right)\geq c. 
\] 
This proves \eqref{eq:Borel} and completes the proof of the lemma. 
\end{proof}

As a corollary of Lemma~\ref{l:hitting1} we obtain the following lemma. 
Let $\mu^{(i)}$, $i\in\{1,\ldots,s_d-1\}$ be independent Poisson point measures with distribution $\mathrm{Pois}(u,W^*)$, 
where $s_d$ is defined in \eqref{def:sd}. 
Let $\mathbb P$ be their joint law.
We construct the graph $G' = (V',E')$ as follows. The set of vertices $V'$ is the set of trajectories from $\cup_{i=1}^{s_d-1}\mathrm{Supp}(\mu^{(i)})$, 
and the set of edges $E'$ is the set of pairs of different trajectories from $\cup_{i=1}^{s_d-1}\mathrm{Supp}(\mu^{(i)})$ that intersect. 
\begin{lemma}\label{l:interlacement}
Let $d\geq 5$ and $u > 0$. 
Then, with the above notation, 
\[
\mathbb P(\mathrm{diam}(G') \leq s_d) = 1 .\
\]
\end{lemma}
\begin{proof}
Take a positive integer $r$. 
By Definition~\ref{def:RIrestricted} (see also the notation there), 
for each $i\in\{1,\ldots,s_d-1\}$, $\mu^{(i)} = \mu^{(i)}_r + \mu^{(i)}_{r,\infty}$, and
the measures $\mu^{(i)}_r$ and $\mu^{(i)}_{r,\infty}$ are independent by property (5) of $\mathrm{Pois}(u,W^*)$. 
For any $i\in\{1,\ldots,s_d-1\}$, let $N^{(i)}$ be the number of trajectories in $\mathrm{Supp}(\mu^{(i)}_r)$. 
In other words, $N^{(i)}$ is the number of doubly-infinite trajectories modulo time-shift from $\mathrm{Supp}(\mu)$ that intersect $B(r)$. 
By property (1) of $\mathrm{Pois}(u,W^*)$, $N^{(i)}$ has the Poisson distribution with parameter $u\mathrm{cap}(B(r))$. 
By the definition of $\mathrm{Pois}(u,W^*)$, we know that (recall the notation from Section~\ref{sec:model}), 
for each $i\in\{1,\ldots,s_d-1\}$, $\mu^{(i)}_r = \sum_{j=1}^{N^{(i)}}\delta_{\pi^*(X^{(i)}_j)}$, where 
$X^{(i)}_1,\ldots,X^{(i)}_{N^{(i)}}$ are doubly-infinite trajectories from $W$ such that 
(a) they are parametrized in such a way that $X^{(i)}_j(0)\in B(r)$ and $X^{(i)}_j(t)\notin B(r)$ for all $t<0$ and for all $j\in\{1,\ldots,N^{(i)}\}$, 
and (b) they satisfy properties (2) and (3) of $\mathrm{Pois}(u,W^*)$.  
In particular, given $N^{(i)}$ and $(X^{(i)}_j(0))_{j=1}^{N^{(i)}}$, 
the forward trajectories $(X^{(i)}_j(t),t\geq 0)_{j=1}^{N^{(i)}}$ are distributed as independent simple random walks. 

Property (5) of $\mathrm{Pois}(u,W^*)$ gives that for each $i\in\{1,\ldots,s_d-1\}$, 
all the random walks $(X^{(i)}_j(t),t\geq 0)_{j=1}^{N^{(i)}}$ are independent from $\mu^{(k)}_{r,\infty}$ for $k\in\{1,\ldots,s_d-1\}$.
Therefore, Lemma~\ref{l:hitting1} and Remark~\ref{rem:distance} imply that, 
given $N^{(i)}$ and $(X^{(i)}_j(0))_{j=1}^{N^{(i)}}$ for all $i\in\{1,\ldots,s_d-1\}$, almost surely, for each pair of different random walks 
$(X^{(i)}_j(t),t\geq 0)$ and $(X^{(k)}_l(t),t\geq 0)$, 
there exist doubly-infinite trajectories $w_m\in\mathrm{Supp}(\mu^{(m)}_{r,\infty})$, $1\leq m\leq s_d-1$, 
such that $X^{(i)}_j\cap w_1\neq \emptyset$, $X^{(k)}_l\cap w_{s_d-1}\neq\emptyset$, and $w_i\cap w_{i+1}\neq \emptyset$ for $i\in\{1,\ldots,s_d-2\}$.
Since this holds for any $r$, the result follows. 
\end{proof}

\begin{proof}[Proof of Theorem~\ref{thm:diameter}: upper bound on diameter]
We complete the proof of Theorem~\ref{thm:diameter} by showing that $\mathbb P(\mathrm{diam}(G)\leq s_d) = 1$. 
By Remark~\ref{rem:d34}, we may and will assume that $d\geq 5$. 
Let $\mu^{(1)},\ldots,\mu^{(s_d-1)}$ be independent Poisson point measures on $W^*$ with distribution $\mathrm{Pois}(u/(s_d-1),W^*)$. 
We construct the graph $G' = (V',E')$ as follows. The set of vertices $V'$ is the set of trajectories from $\cup_{i=1}^{s_d-1}\mathrm{Supp}(\mu^{(i)})$, 
and the set of edges $E'$ is the set of pairs of different trajectories from $\cup_{i=1}^{s_d-1}\mathrm{Supp}(\mu^{(i)})$ that intersect. 
Lemma~\ref{l:interlacement} implies that the diameter of $G'$ is at most $s_d$. 
On the other hand, by property (4) of $\mathrm{Pois}(u,W^*)$, graphs $G$ and $G'$ have the same law. 
This completes the proof. 
\end{proof}

\bigskip
\textbf{Acknowledgments.}  We would like to thank A.-S. Sznitman for suggesting to look for an alternative proof of the connectivity of the random interlacement, inspiring discussions, and valuable comments. We also thank A. Drewitz for comments on the manuscript.

\end{document}